\documentclass[12pt]{article}
\usepackage{amssymb,fullpage,amsmath}
\usepackage{tikz}

\newcommand{\ceil}[1]{\left \lceil #1 \right \rceil}

\newcommand{\size}[1]{\left \vert #1 \right \vert}
\newcommand{\ttone}[1]{\tau_{#1}}

\def\qed{\hfill\rule{2mm}{2mm}\par\bigskip}
\def\PF{\noindent{\it Proof.}}

\newtheorem{theorem}{Theorem}[section]

\newtheorem{lemma}[theorem]{Lemma}

\newtheorem{proposition}[theorem]{Proposition}
\newtheorem{definition}[theorem]{Definition}
\newtheorem{conjecture}[theorem]{Conjecture}
\newtheorem{theorema}{Theorem}
\newenvironment{proof}{\PF}{\qed}

\begin{document}

\title{New results in $t$-tone coloring of graphs}
\author{
Daniel W. Cranston\thanks{Mathematics Dept., Virginia Commonwealth University,
dcranston@vcu.edu},
Jaehoon Kim\thanks{Mathematics Dept., University of Illinois, kim805@illinois.edu.  
Research partially supported by the Arnold O. Beckman Research Award of the 
University of Illinois at Urbana-Champaign.},
and William B. Kinnersley\thanks{Mathematics Dept., University of Illinois,
wkinner2@illinois.edu.  Research partially supported by
NSF grant DMS 08-38434, ``EMSW21-MCTP: Research Experience for Graduate
Students''.}
}

\maketitle

\begin{abstract}
A $t$-tone $k$-coloring of $G$ assigns to each vertex of $G$ a set of $t$
colors from $\{1, \ldots, k\}$ so that vertices at distance $d$ share fewer
than $d$ common colors.  The {\it $t$-tone chromatic number} of $G$, denoted
$\tau_t(G)$, is the minimum $k$ such that $G$ has a $t$-tone $k$-coloring. 
Bickle and Phillips showed that always $\tau_2(G) \le [\Delta(G)]^2 +
\Delta(G)$, but conjectured that in fact $\tau_2(G) \le 2\Delta(G) + 2$; we
confirm this conjecture when $\Delta(G) \le 3$ and also show that always
$\tau_2(G) \le \ceil{(2 + \sqrt{2})\Delta(G)}$.  For general $t$ we prove that
$\tau_t(G) \le (t^2+t)\Delta(G)$.  
Finally, for each $t\ge 2$ we show that there exist constants $c_1$ and
$c_2$ such that for every tree $T$ we have $c_1 \sqrt{\Delta(T)} \le \tau_t(T)
\le c_2\sqrt{\Delta(T)}$.
\end{abstract}

\section{Introduction}

In standard vertex coloring, we give colors to the vertices of a graph so
that adjacent vertices get distinct colors.  Several variants of graph
coloring place restrictions on the colors of vertices that are near each
other, but not necessarily adjacent.  In a {\em distance-$k$ coloring}, any 
vertices within distance $k$ of each other must receive distinct colors.  In
an {\em injective coloring}, any two vertices with a common neighbor must
receive distinct colors, but adjacent vertices need not.  
In an {\em $L(2,1)$-labeling} 
each vertex receives a nonnegative integer as its label, such that the labels
on adjacent vertices differ by at least 2 and those on vertices at distance 2
differ by at least 1. 

Bickle and Phillips~\cite{BP} introduced the related notion of {\em $t$-tone
coloring}.  Intuitively, a {\em $t$-tone $k$-coloring} of $G$ assigns to each
vertex of $G$ a set of $t$ colors from $\{1, \ldots, k\}$ so that vertices at
distance $d$ share fewer than $d$ common colors.  This notion is especially appealing when $t = 2$.  In this case, each vertex receives a set of two colors; adjacent vertices receive disjoint sets and vertices at distance 2 receive distinct sets.

Before giving a formal definition, we first establish some basic notation and
terminology.  We write $[k]$ as shorthand for $\{1, \ldots, k\}$ and denote by
${[k] \choose t}$ the family of $t$-element subsets of $[k]$.  We denote the
distance between vertices $u$ and $v$ by $d(u,v)$.  Vertices $u$ and $v$ are
{\em neighbors} if $d(u,v)=1$ and {\em second-neighbors} if $d(u,v) = 2$.

\begin{definition}\label{def_ttone}\cite{BP}
Let $G$ be a graph and $t$ a positive integer.  A {\em $t$-tone $k$-coloring} of $G$ is a function $f: V(G) \rightarrow {[k] \choose t}$ such that $\size{f(u) \cap f(v)} < d(u,v)$ for all distinct vertices $u$ and $v$.  A graph that has a $t$-tone $k$-coloring is {\em $t$-tone $k$-colorable}.  The {\em $t$-tone chromatic number} of $G$, denoted $\tau_t(G)$, is the minimum $k$ such that $G$ is $t$-tone $k$-colorable.
\end{definition}

Given a $t$-tone coloring $f$ of $G$, we call $f(v)$ the {\em label} of $v$ and
the elements of $[k]$ {\em colors}.  When the meaning is clear, we
omit set notation from labels; that is, we denote the label $\{a, b\}$ by $ab$.
 Note that for each $t$, the parameter $\tau_t$ is monotone: when $H$ is a
subgraph of $G$, every $t$-tone $k$-coloring of $G$ restricts to a $t$-tone
$k$-coloring of $H$, so $\tau_t(H) \le \tau_t(G)$.

Bickle and Phillips \cite{BP} established several basic results on $t$-tone
coloring,  many of which focused on the relationship between
$\tau_2(G)$ and $\Delta(G)$.  By looking at proper colorings of the graph
$G^2$, they proved that always $\tau_2(G) \le [\Delta(G)]^2 + \Delta(G)$. 
However, they conjectured that this bound is far from tight:
\begin{conjecture}\cite{BP}\label{conj_maxdegree}
If $G$ is a graph with maximum degree $r$, then $\tau_2(G) \le 2r + 2$.  If
$r\ge 3$, then equality holds only when $G$ contains $K_{r+1}$.  
\end{conjecture}
When $G$ is 3-regular, they posed the following stronger conjecture:
\begin{conjecture}\cite{BP}\label{conj_cubic}
If $G$ is a 3-regular graph, then:
\begin{enumerate}
\item[(a)] $\tau_2(G) \le 8$;
\item[(b)] $\tau_2(G) \le 7$ when $G$ does not contain $K_4$;
\item[(c)] $\tau_2(G) \le 6$ when $G$ does not contain $K_4 - e$.
\end{enumerate}
\end{conjecture}
\noindent
Since they also characterized all 2-tone 5-colorable 3-regular graphs, this conjecture would yield a complete characterization of the 2-tone chromatic numbers of 3-regular graphs.

In Section \ref{2tone}, we focus on 2-tone colorings, with an eye toward
proving Conjectures~\ref{conj_maxdegree} and~\ref{conj_cubic}. 
As progress toward Conjecture \ref{conj_maxdegree}, we give a short proof that
always $\tau_2(G) \le \ceil{(2 + \sqrt{2})\Delta(G)}$.  Simple modifications of
this argument yield better bounds when $G$ is bipartite or chordal.  We next
refute part (b) of Conjecture \ref{conj_cubic} by showing that the Heawood
graph has 2-tone chromatic number 7.  Finally, our main result in
Section~\ref{2tone} confirms part (a) of Conjecture \ref{conj_cubic}:
\begin{theorema}
If $G$ is a graph with $\Delta(G) \le 3$, then $\tau_2(G) \le 8$.
\end{theorema}

In Section \ref{ttone}, we consider $t$-tone colorings for general $t$.  Our
main result is: 
\begin{theorema}
For each $t$ there exists a constant $c=c(t)$ such that $\tau_t(T) \le
c\sqrt{\Delta(T)}$ whenever $T$ is a tree, and this bound is asymptotically
tight.  
\end{theorema}
For general graphs, our best bound is $\tau_t(G) \le (t^2 + t)\Delta(G)$.  This
result implies that, for fixed $\Delta(G)$, we have $\tau_t(G) \le ct^2$ for
some constant $c$.  The asymptotics of this bound are near-optimal with respect
to $t$, since for each $r \ge 3$ there exist a constant $c$ and graphs $G_t$ such that $\Delta(G_t) = r$ and $\tau_t(G_t) \ge ct^2/\lg t$.  Finally, when $G$ has degeneracy at most $k$, we prove $\tau_t(G) \le kt + kt^2[\Delta(G)]^{1-1/t}$.


\section{2-tone Coloring}\label{2tone}

In this section we focus on 2-tone coloring.  We first attack Conjecture
\ref{conj_maxdegree}.  It was shown in \cite{BP} that always $\tau_2(G) \le
[\Delta(G)]^2 + \Delta(G)$; we improve this result by giving an upper bound on
$\tau_2(G)$ that is linear in
$\Delta(G)$, rather than quadratic.  This proof---along with several others
throughout the paper---proceeds by building a $t$-tone coloring of a graph
iteratively, coloring one vertex at a time. 
  
\begin{definition}
A {\em partial $t$-tone $k$-coloring} of a graph $G$ is a function $f: S \rightarrow {[k] \choose t}$, with $S \subseteq V(G)$, such that $\size{f(u) \cap f(v)} < d(u,v)$ whenever $u,v \in S$.  Vertices not in $S$ are {\em uncolored}.  An {\em extension} of $f$ to an uncolored vertex $v$ is a partial coloring $f^\prime$ that assigns a label to $v$ but otherwise agrees with $f$. 
\end{definition}

It is important to note that a $t$-tone $k$-coloring of a subgraph $H$ of $G$ need not be a partial $t$-tone $k$-coloring of $G$, since the distance between two vertices may be smaller in $G$ than in $H$.

\begin{theorem}\label{2tone_maxdegree}
For every graph $G$, we have $\tau_2(G) \le \ceil{(2 + \sqrt{2})\Delta(G)}$.
\end{theorem}
\begin{proof}
Let $k = \ceil{(2 + \sqrt{2})\Delta(G)}$ and let $V(G) = \{v_1, \ldots, v_n\}$. 
Starting with all vertices uncolored, we extend our partial coloring to $v_1, v_2, \ldots, v_n$ in order. 
When extending to $v_i$, we need only enforce two constraints.  First, the label
on $v_i$ cannot contain any color appearing on $v_i$'s neighbors; there
remain at least $\ceil{\sqrt{2}\Delta(G)}$ other colors, so at least ${\sqrt{2}\Delta(G) \choose 2}$ labels are available.  Next, the label on $v_i$ cannot appear on any second-neighbor of $v_i$; this condition forbids at most  
$\Delta(G)(\Delta(G) - 1)$ labels.  Since
${\sqrt{2}\Delta(G) \choose 2} > \Delta(G)(\Delta(G)-1)$, some label remains
for use on $v_i$.
\end{proof}

Similar approaches yield tighter bounds on $\tau_2(G)$ for bipartite graphs and chordal graphs.

\begin{proposition}
If $G$ is a bipartite graph, then $\tau_2(G) \le 2\ceil{\sqrt{2}\Delta(G)}$.
\end{proposition}
\begin{proof}
A {\em palette} is a set of colors; we construct a 2-tone coloring of $G$ using
two disjoint palettes, each of size $\ceil{\sqrt{2}\Delta(G)}$.  We assign each
partite set its own palette and color the vertices in each set using only
colors from its palette.  Since adjacent vertices are assured disjoint labels, 
it suffices to ensure that vertices at distance 2 receive distinct labels. 

We color each partite set independently.  Within a partite set, we order the
vertices arbitrarily and color iteratively.  Each vertex $v$ has at most
$\Delta(G)(\Delta(G)-1)$ second-neighbors.  Since each palette admits ${\sqrt{2}\Delta(G) \choose 2}$ labels, we may always extend a partial coloring to $v$.
\end{proof}

A {\em simplicial elimination ordering} of a graph $G$ is an ordering $v_1, \ldots, v_n$ of $V(G)$ such that the later neighbors of each vertex form a clique; it is well-known that chordal graphs are precisely those graphs having simplicial elimination orderings.

\begin{proposition}
If $G$ is a chordal graph, then $\tau_2(G) \le \ceil{(1 + \sqrt{6}/2)\Delta(G)} + 1$.  
\end{proposition}
\begin{proof}
Let $k = \ceil{(1 + \sqrt{6}/2)\Delta(G)} + 1$.  Let $v_1, \ldots, v_n$ be the
reverse of a simplicial elimination ordering of $G$.
We construct a 2-tone
$k$-coloring of $G$ by coloring iteratively with respect to this ordering.

Suppose we want to color $v_i$.  Let $S$ be the set of
earlier neighbors of $v_i$, and let $d = \size{S}$.  
If $v_j$ is a later neighbor of $v_i$, then
by our choice of ordering, all earlier neighbors of $v_j$ are adjacent to $v_i$.
Hence every earlier second-neighbor of $v_i$ is adjacent to some vertex in $S$.
Each vertex in $S$ is adjacent to $v_i$ itself along with the other $d-1$
vertices of $S$.  Hence $v_i$ has at most $d(\Delta(G)-d)$ earlier second-neighbors.

As many as $2d$ colors may appear on $S$, so at least
$k-2d$ colors remain.  We have ${k-2d \choose 2}$ labels using these colors,
so we need ${k-2d \choose 2} > d(\Delta(G)-d)$.  
Straightforward computation shows that this inequality holds whenever $k\ge
\ceil{(1+\sqrt{6}/2)\Delta(G)}+1$.
\end{proof}

\begin{proposition}
For every $\epsilon > 0$, there exists an $r_0$ such that whenever $r > r_0$, if $G$ is a chordal graph with maximum degree $r$, then $\tau_2(G) \le (2 + \epsilon)r$.
\end{proposition}
\begin{proof}
Let $G$ be a chordal graph with maximum degree $r$.  Kr\'al \cite{K} showed that, for some constant $c$, the graph $G^2$ is $cr^{3/2}$-degenerate.  Thus, there is some ordering $v_1, \ldots, v_n$ of $V(G)$ such that each vertex has at most $cr^{3/2}$ earlier second-neighbors.  Let us color iteratively with respect to this ordering using $k + 2r$ colors, for some $k$ to be specified later.  When coloring $v_i$, as many as $2r$ colors may appear on its neighbors; at least $k$ other colors remain.  Thus we may color $v_i$ so long as it has fewer than ${k \choose 2}$ earlier second-neighbors; taking $k \ge \sqrt{2c}r^{3/4} + 1$ suffices.  Hence $\tau_t(G) \le 2r + \sqrt{2c}r^{3/4} + 1$, from which the claim follows.
\end{proof}


We next turn our attention to 3-regular graphs and Conjecture \ref{conj_cubic}.  
Later in this section, we prove part (a) of Conjecture \ref{conj_cubic} by showing that $\tau_2(G) \le 8$ whenever $\Delta(G) \le 3$; first we disprove part (c) by showing that the Heawood Graph, which has girth 6, has 2-tone chromatic number 7.

\begin{theorem}
The Heawood Graph is not 2-tone 6-colorable.
\end{theorem}
\begin{proof}
Let $G$ denote the Heawood Graph.
Recall that $G$ is the incidence graph of the Fano Plane; thus it is bipartite,
and every two distinct vertices in the same partite set have exactly one common
neighbor (and hence lie at distance 2).  Call a 2-tone 6-coloring of $G$ a {\em good coloring}.  For distinct colors $a,b,c,d$, call the set of labels $\{ab, cd, ac, bd\}$ a {\em complementary pair}.  For distinct colors $a,b,c,d,e,f$, call the set of labels $\{ab, cd, ef\}$ a {\em disjoint triple}.
Let $A$ and $B$ denote the partite sets of $G$.

We prove three claims: {\bf (1)} No good coloring uses all four labels in a
complementary pair on vertices in the same partite set; {\bf (2)} No good
coloring uses all three labels in a disjoint triple on vertices in the same
partite set; {\bf (3)} For any subset $L$ of ${[6]\choose 2}$ with $|L|=7$,
either $L$
contains a complementary pair or it contains a disjoint triple.
The theorem immediately follows from these claims by supposing $G$
has a good coloring and letting $L$ be the set of labels used on $A$.

{\bf (1)} Suppose instead that the claim is false.  By symmetry, labels $12, 34,
13,$ and $24$ all appear on vertices in $A$.  The common neighbor of the
vertices labeled $12$ and $34$ must receive label $56$, as must the common
neighbor of the vertices labeled $13$ and $24$.  Since $G$ is 3-regular, the two
vertices labeled $56$ are distinct; since they lie at distance 2, the coloring
is invalid.

{\bf (2)} Suppose instead that the claim is false.  By symmetry, labels $12,
34,$ and $56$ all appear on vertices in $A$.  These vertices
cannot all have a common neighbor $u$, since then $u$ would have no valid
label.  Thus they lie on a 6-cycle, and the three vertices of
this 6-cycle in $B$ must also have labels $12, 34,$ and $56$.  

Consider a vertex $v\in A$ not
adjacent to any vertex of this 6-cycle. (There is exactly one such vertex.)
The label on $v$ cannot be $12, 34,$ or $56$, so without loss of
generality, it is $13$.  The common neighbor of $v$ and the vertex in $A$ having label
$56$ must have label $24$, and the common neighbor of this vertex and the 
vertex in $B$ having label $56$ must have label $13$.  So two vertices in 
$A$ have label $13$; they must be distinct, since only one is
adjacent to a vertex on the 6-cycle. 
Since they lie at distance 2, the coloring is invalid.

{\bf (3)} 
Consider a color appearing in the most elements of $L$; without loss of
generality, this color is 1.  Let $L_1$ be the set of labels in
$L$ that contain
1.  Note that $3 \le \size{L_1} \le 5$.  We consider 3 cases.

If $\size{L_1} = 5$, then exactly two labels in $L$ do not
appear in $L_1$.  If these labels are disjoint, then $L$
contains a disjoint triple; otherwise, $L$ contains a complementary pair.

If $\size{L_1} = 4$, then without loss of generality $L_1 =
\{12, 13, 14, 15\}$.  If two labels in $L$ contain 6, then
$L$ contains a complementary pair.  Similarly, if $L -
L_1$ contains two non-disjoint labels not using 6, then $L$
contains a complementary pair.  Thus we may suppose that $L -
L_1$ contains two disjoint labels not using 6 and one label using 6.
Now the label using 6 is disjoint from one of the labels not using 6; these two
labels, together with some label from $L_1$, form a disjoint triple.

If $\size{L_1} = 3$, then without loss of generality $L_1 =
\{12, 13, 14\}$.  Let $S_1 = \{23, 24, 34\}$, let $S_2 = \{25, 35, 45\}$, and
let $S_3 = \{26, 36, 46\}$.  If $L$ contains two or more labels from
any single $S_i$, then these labels, together with two labels from
$L_1$, form a complementary pair.  Thus we may suppose $L$
contains exactly one label from each $S_i$ and also contains the label $56$.
Now the label in $L \cap S_1$, the label $56$, and some element of
$L_1$ form a disjoint triple.
\end{proof}

Below we give a 2-tone 7-coloring of the Heawood graph, which completes the proof that its 2-tone chromatic number is 7.

\begin{center}
\begin{tikzpicture}[thick,scale=2]
\tikzstyle{blavert}=[circle, draw, fill=black, inner sep=0pt, minimum width=5pt]
\tikzstyle{redvert}=[circle, draw, fill=red, inner sep=0pt, minimum width=5pt]
\draw \foreach \i in {2, 4, ..., 14}
{
(\i*360/14:1) node[blavert]{} -- (\i*360/14+360/14:1)
(\i*360/14+360/14:1) node[redvert]{} -- (\i*360/14+720/14:1)
(\i*360/14:1)  -- (\i*360/14+5*360/14:1)
};
\draw \foreach \i/\lab in 
{1/14, 2/25, 3/36, 4/47, 5/51, 6/62, 7/73, 8/14, 9/25, 10/36, 11/47, 12/51, 13/62, 14/73}
{
(\i*360/14:1.2) node[]{$\lab$} 
};

\draw (0,-1.5) node {Fig.~1: A 2-tone 7-coloring of the Heawood graph.};
\end{tikzpicture}
\end{center}

We next show that $\tau_2(G) \le 8$ whenever $\Delta(G) \le 3$, thus verifying part (a) of Conjecture \ref{conj_cubic}.  The proof requires careful attention to detail, so we isolate some of the more delicate arguments in lemmas.  Before stating the lemmas, we introduce some terminology.

\begin{definition}
Let $f$ be a partial $2$-tone coloring of a graph $G$ and let $v$ be an uncolored vertex.  A {\em valid} label for $v$ is a label by which $f$ can be extended to $v$.  A {\em free} color at $v$ is one not appearing on any neighbor of $v$.  A {\em candidate} label for $v$ is a label containing only free colors.  An {\em obstruction} of $v$ is a candidate label that is not valid (because it appears on some second-neighbor of $G$).
\end{definition}

Our first lemma is short and simple, but provides a good introduction to the techniques that appear throughout the proof.

\begin{lemma}\label{cubic_extend}
Let $G$ be a graph with maximum degree at most 3.  Let $f$ be a partial 2-tone 8-coloring of $G$ and let $v$ be an uncolored vertex.  If $v$ has at least one uncolored neighbor and at least one uncolored second-neighbor, then $f$ can be extended to $v$.  
\end{lemma}
\begin{proof}
At least four colors are free at $v$, so it has at least six candidate labels. 
Since $v$ has an uncolored second-neighbor, $v$ has at most five obstructions, so some candidate is valid.  
\end{proof}

In the main proof we first color all vertices except for those on some induced
cycle $C$; we then iteratively extend our partial coloring along $C$.  We
will need to maintain some flexibility while doing so, and the next two lemmas
provide this desired freedom.

\begin{lemma}\label{cubic_freedom_0}
Let $G$ be a 3-regular graph, let $v$ be a vertex of $G$, and let $w_1$ and $w_2$ be distinct neighbors of $v$.  Let $f$ be a partial coloring of $G$ that leaves $v$, $w_1$, and $w_2$ uncolored, and let $f_1$ and $f_2$ be distinct extensions of $f$ to $w_1$.  If two second-neighbors of $v$ do not yield obstructions under any $f_i$, then some $f_i$ can be extended to $v$ in three different ways.
\end{lemma}
\begin{proof}
Let $S_i$ be the set of free colors at $v$ under $f_i$.  Under each $f_i$, at
most four colors appear on neighbors of $v$, so $\size{S_i} \ge 4$.  Either
some $S_i$ contains at least five colors, or $S_1 \not = S_2$; in either case,
the $f_i$ yield at least nine candidate labels between them.  Since $v$ has at
most four obstructions, the two $f_i$ together yield at least five valid
labels, so by the Pigeonhole Principle 
some $f_i$ admits three extensions to
$v$.
\end{proof}
\begin{lemma}\label{cubic_freedom}
Let $G$ be a 3-regular graph, let $v$ be a vertex of $G$, and let $w_1$ and $w_2$ be distinct neighbors of $v$.  Let $f$ be a partial coloring of $G$ that leaves $v$, $w_1$, and $w_2$ uncolored, and let $f_1$, $f_2$, and $f_3$ be distinct extensions of $f$ to $w_1$.  If some second-neighbor of $v$ does not yield an obstruction under any $f_i$, then some $f_i$ can be extended to $v$ in three different ways.
\end{lemma}
\begin{proof}
Let $S_i$ be the set of colors free at $v$ under $f_i$.  Under each $f_i$, at most four colors appear on neighbors of $v$, so $\size{S_i} \ge 4$.  If some $S_i$ contains five or more colors, then $v$ has at least ten candidate labels and at most five obstructions under $f_i$, so $f_i$ admits at least five extensions to $v$.  Otherwise, since the $f_i$ assign different labels to $w_1$, no two $S_i$ are the same.  Since $v$ has at least six candidate labels under each $f_i$, it suffices to show that $v$ cannot have four obstructions under each $f_i$ simultaneously.

Without loss of generality, $S_1 = \{1, 2, 3, 4\}$.  Since $S_2 \not = S_1$, we may assume $5 \in S_2$.  If additionally $S_2$ contains some other color not in $S_1$, then at most one label is a candidate under both $f_1$ and $f_2$; in this case $v$ has at most one common obstruction under $f_1$ and $f_2$, so it cannot have four obstructions under both $f_1$ and $f_2$.  Hence we may assume $S_2 = \{1, 2, 3, 5\}$.  Now $f_1$ and $f_2$ yield three common candidates, namely $12, 13,$ and $23$; if $v$ does not have three valid labels under either $f_i$, then all three common candidates must be obstructions.  Moreover, of the two remaining obstructions, one is in $\{14, 24, 34\}$ and the other is in $\{15, 25, 35\}$.  If $S_3$ contains $1, 2,$ and $3$, then without loss of generality $S_3 = \{1, 2, 3, 6\}$, and $f_3$ can be extended via $16$, $26$, and $36$.  Otherwise at most one of $12$, $13$, and $23$ is an obstruction under $f_3$, and again $f_3$ admits three extensions to $v$. 
\end{proof}

Our final lemma helps us leverage the flexibility ensured by Lemma \ref{cubic_freedom} to complete a partial coloring.
 
\begin{lemma}\label{cubic_finish_1}
Let $G$ be a 3-regular graph.  Let $v$ be a vertex of $G$, let $w_1,w_2,$ and $w_3$ be its neighbors, and let $x$ be one of its second-neighbors.  Let $f$ be a partial coloring of $G$ that leaves $v$ and $w_1$ uncolored, and under which $w_2$ shares one color with $w_3$ and one with $x$.  If $f$ has three extensions to $w_1$, then one of these extensions can itself be extended to $v$.
\end{lemma}
\begin{proof}
Let $f_1$, $f_2$, and $f_3$ be extensions of $f$ to $w_1$.  Since $w_2$ and $x$ share a color, $x$ cannot yield an obstruction of $v$, so $v$ has at most five different obstructions between all three $f_i$.  Since $w_2$ and $w_3$ share a color, at most five colors appear on neighbors of $v$ in each $f_i$, hence always at least three colors are free at $v$.  Let $S_i$ be the set of free colors at $v$ under $f_i$.  If any $S_i$ contains at least four colors, then $v$ has at least six candidate labels under $f_i$, one of which must be valid.  Otherwise, each $S_i$ has size three; moreover, since the $f_i$ differ in the colors they assign to $w_1$, no two $S_i$ are identical.  $S_1$ and $S_2$ together yield at least five different candidate labels for $v$, and $S_3$ yields a sixth; again we have six candidate labels, one of which must be valid.  Thus some $f_i$ can be extended to $v$.
\end{proof}

We are now ready to present the main proof.

\begin{theorem}
If $G$ is a graph with $\Delta(G)\le 3$, then $\tau_2(G) \le 8$.
\end{theorem}
\begin{proof}
Suppose otherwise, and let $G$ be a smallest counterexample.  Clearly $G$ is connected and is not $K_4$.

Suppose that $G$ is not 3-regular, and let $v$ be a vertex of degree 1 or 2.  By Lemma \ref{cubic_extend}, iteratively coloring in non-increasing order of distance from $v$ yields a partial 2-tone 8-coloring of $G$ leaving only $N[v]$ uncolored.  Each neighbor $u$ of $v$ now has at least four free colors (hence at least six candidate labels) and at most five second-neighbors, so we may extend the coloring to $u$.  Likewise, $v$ itself now has at least four free colors and at most four second-neighbors, so we may extend to $v$ as well, completing the coloring and contradicting the choice of $G$.  Hence $G$ must be 3-regular.

Next suppose that $G$ contains an induced $K_{2,3}$.  Let $x_1, x_2, y_1, y_2$,
and $y_3$ be the vertices of this $K_{2,3}$, with the $x_i$ the vertices of
degree 3 and the $y_i$ the vertices of degree 2; let $u_i$ be the third
neighbor of each $y_i$.  Let $G^\prime = G - \{x_1, x_2, y_1, y_2, y_3\}$. 
Since $G^\prime$ is not 3-regular, it has a 2-tone 8-coloring, which is also a
partial 2-tone 8-coloring of $G$.  Without loss of generality, the color 1 does
not appear on any $u_i$.  We aim to color each $y_i$ with a label containing
color 1; each $y_i$ has five such candidate labels and at most four
second-neighbors, so this is possible.  Now each $x_i$ has at least four free
colors, and hence at least six candidate labels.  Since each $x_i$ has at most
four second-neighbors, we may extend the coloring to each $x_i$ in turn, again
contradicting the choice of $G$.  Thus $G$ is $K_{2,3}$-free.

Let $C$ be a shortest cycle in $G$; label its vertices $v_1, \ldots, v_k$ in
cycle order.  Let $u_1, \ldots, u_k$ be the neighbors off $C$ of $v_1, \ldots,
v_k$, respectively.  The $u_i$ need not be distinct, but (since $G \not = K_4$)
cannot all be the same vertex.  If $C$ is a triangle, then without loss of
generality $u_3 \not = u_1$.  If not, then for all $i$ we have $u_{i-1} \not =
u_{i+1}$: if $C$ is a four-cycle then this follows from the fact that $G$ is
$K_{2,3}$-free, and otherwise it follows from the minimality of $C$.  In any
case, construct $G^\prime$ from $G$ by deleting the vertices of $C$ and adding
the edge $u_{k-1}u_{1}$ (if it is not already present); if $C$ is not a
triangle, then add the edge $u_ku_2$ as well.  By the minimality of $G$, the
graph $G^\prime$ is 2-tone 8-colorable.  A 2-tone 8-coloring of $G^\prime$ is
also a
partial 2-tone 8-coloring of $G$ in which only the $v_i$ are uncolored and in
which $u_{k-1}$ and $u_1$ have disjoint labels; if $C$ has at least four
vertices, then also $u_k$ and $u_2$ have disjoint labels.  We use such a
coloring as a starting point in producing a 2-tone 8-coloring of $G$.  
 
We have three cases to consider.  {\bf (1)} If the label on $u_k$ is identical
to one of the labels on $u_{k-1}$ or $u_{1}$, then by symmetry we may suppose
that $u_{k-1}$, $u_k$, and $u_{1}$ have labels $12, 12$, and $34$.  {\bf (2)}
If the label on $u_k$ is disjoint from the labels on $u_{k-1}$ and $u_{1}$,
then we may suppose that $u_{k-1}$, $u_k$, and $u_{1}$ have labels $12, 34$,
and $56$.  {\bf (3)} Otherwise, we may suppose that $u_{k-1}$, $u_k$, and
$u_{1}$ have labels $12, 13$, and $L$, where $1 \not \in L$.

{\bf Case (1):} {\em $u_{k-1}$, $u_k$, and $u_{1}$ have labels 12, 12, 34.}  We
aim to assign $v_1$ a label containing either 1 or 2; $v_1$ has nine such
candidate labels, and it has at most four obstructions, so at least five such
labels are valid.  Since we have at least three ways to extend to $v_1$, by
Lemma \ref{cubic_freedom}, we subsequently have at least three ways to extend
to $v_2$, then to $v_3$, and so on up to $v_{k-2}$.  Since the labels on
$u_{k-1}$ and $v_1$ have nonempty intersection, $v_1$ cannot yield an
obstruction of $v_{k-1}$, so again we have three ways to extend to $v_{k-1}$. 
Now applying Lemma \ref{cubic_finish_1} (with $v = v_k, w_1 = v_{k-1}, w_2 =
v_1, w_3 = u_k,$ and $x = u_{k-1}$) lets us complete the coloring.  

{\bf Case (2):} {\em $u_{k-1}$, $u_k$, and $u_{1}$ have labels 12, 34, 56.} 
First suppose that $C$ is a triangle.  Give $v_1$ a label from $\{13, 23, 37,
38\}$; since $v_1$ has at most two obstructions, this is possible.  Next give
$v_2$ a label from $\{45, 46, 47, 48\}$; at most one of these labels has
nonempty intersection with the label on $v_1$, and $v_2$ has at most two
additional obstructions, so again some such label is valid.  We have ensured
that four colors remain free at $v_3$.  Thus $v_3$ has six candidate labels and
at most four obstructions, so we can complete the coloring.

Suppose now that $C$ is not a triangle.  We aim to assign $v_1$ a label from
$\{13, 14, 23, 24\}$.  Although $v_1$ has four colored second-neighbors, $u_k$ has label $34$, which is not an obstruction.  Moreover, by construction the label on $u_2$ contains neither 3 nor 4, so it also cannot be an obstruction.  Thus, at least two such labels are valid.  By Lemma \ref{cubic_freedom_0}, this coloring admits three extensions to $v_2$.  Now we may apply Lemma \ref{cubic_freedom} and Lemma \ref{cubic_finish_1} (with $v = v_k, w_1 = v_{k-1}, w_2 = v_1, w_3 = u_k,$ and $x = u_{k-1}$) as before to complete the coloring.

{\bf Case (3):} {\em $u_{k-1}$, $u_k$, and $u_{1}$ have labels $12, 13, L$,
where $1 \not \in L$.}  We aim to give $v_1$ a label containing either 1 or 3. 
If $3 \not \in L$, then $v_1$ has at least nine such candidates and at most
four obstructions, so at least five of the candidates are valid.  Otherwise
$v_1$ has only five such candidate labels, but $u_k$ does not yield an
obstruction, so at least two of these candidates are valid.  In each case, by
Lemma \ref{cubic_freedom_0} we may extend the coloring to $v_2$ in at least
three different ways.  Now by Lemma \ref{cubic_freedom} and Lemma
\ref{cubic_finish_1} (with $v = v_k, w_1 = v_{k-1}, w_2 = u_k, w_3 = v_1,$ and
$x = u_{k-1}$) we can again complete the coloring.
\end{proof}

\section{General $t$-tone Coloring}\label{ttone}

We next study the behavior of $\tau_t$ for general $t$.  We have already noted that $\tau_t(G)$ is monotone in $G$; that is, $\tau_t(H) \le \tau_t(G)$ whenever $H$ is a subgraph of $G$.  It is also true that $\tau_t(G)$ is monotone in $t$.

\begin{proposition}\label{mono_tone}
If $t < t^{\prime}$ and $G$ is any graph, then $\tau_t(G) \le \tau_{t^\prime}(G)$.
\end{proposition}
\begin{proof}
Given a graph $G$ and a $t^\prime$-tone coloring of $G$, we arbitrarily
discard $t^\prime-t$ colors from each label of $G$.  This yields a $t$-tone
coloring, since the process cannot increase the size of the intersection of
any two labels.
\end{proof}

Our first main result in this section is a generalization of Theorem
\ref{2tone_maxdegree}.  In the case $t = 2$, Theorem \ref{2tone_maxdegree} gives
a better bound, since restricting to $t = 2$ allows tighter analysis.

\begin{theorem}\label{general_maxdegree}
For every integer $t$ and every graph $G$, we have $\ttone{t}(G) \le (t^2+t)\Delta(G)$.
\end{theorem}
\begin{proof}
Let $V(G) = \{v_1, v_2, \ldots, v_n\}$, let $r = \Delta(G)$, and let $k = (t^2 + t)r$.  As in the proof of Theorem \ref{2tone_maxdegree}, we construct a $t$-tone $k$-coloring of $G$ by coloring iteratively with respect to the ordering $v_1, \ldots, v_n$.

When coloring $v_i$,  at most $tr$ colors appear on neighbors of $v_i$, so at
least $t^2r$ other colors remain.  We have ${t^2r \choose t}$ labels that use
only these colors, and each is a candidate label for $v_i$.  

Given a label $L$, we say that vertex $u$ {\em forbids} $L$ if $L$ and
the label on $u$ have intersection size at least $d(u,v_i)$.  Recall that we
have already discarded all labels forbidden by neighbors of $v_i$.  For $2\le
d\le t$, each vertex at distance $d$ from $v_i$ forbids at most ${t \choose
d}{t^2r-d \choose t-d}$ labels. 
At most $r(r-1)^{d-1}$ vertices lie at distance $d$ from $v_i$, so to show that
we may color $v_i$, it suffices to show that $$\sum_{d=2}^t {t \choose
d}{t^2r-d \choose t-d}r(r-1)^{d-1} < {t^2r \choose t},$$ or equivalently, that
$$\sum_{d=2}^t \frac{{t \choose d}{t^2r-d \choose t-d}r(r-1)^{d-1}}{{t^2r \choose t}} < 1.$$
Ultimately, we will show that the $d$th term of the sum is no more than $1/d!$,
and thus (since $1/d!\le 2^{1-d}$) the sum is less than 1.
We first simplify each term.  For fixed $d$,
\begin{align*}
\frac{{t \choose d}{t^2r-d \choose t-d}r(r-1)^{d-1}}{{t^2r \choose t}} &= \frac{t!}{d!(t-d)!} \cdot \frac{(t^2r-d)!}{(t-d)!(t^2r-t)!}\cdot r(r-1)^{d-1} \cdot \frac{t!(t^2r-t)!}{(t^2r)!}\\
                                                                   &= \frac{1}{d!} \cdot \left ( \frac{t!}{(t-d)!} \right )^2  \cdot \frac{(t^2r-d)!}{(t^2r)!} \cdot r(r-1)^{d-1}\\
                                                                   &= \frac{1}{d!} \frac{\left (t(t-1)(t-2) \cdots (t-d+1) \right )^2 r(r-1)^{d-1}}{t^2r(t^2r-1)\cdots (t^2r-d+1)}\\
                                                                   &= \frac{1}{d!} \cdot \frac{(t-1)^2(r-1)}{t^2r-1} \cdot \frac{(t-2)^2(r-1)}{t^2r-2} \cdots \frac{(t-d+1)^2(r-1)}{t^2r-d+1}.
\end{align*}
Now for $i$ between $1$ and $d-1$, we have
$$(t-i)^2(r-1) < (t-i)^2r = t^2r - i(2t-i)r \le t^2r - i,$$
hence
$$\frac{{t \choose d}{t^2r-d \choose t-d}r(r-1)^{d-1}}{{t^2r \choose t}} < \frac{1}{d!} \cdot 1 \cdot 1 \cdots 1 = \frac{1}{d!}.$$
Now
$$\sum_{d=2}^t \frac{{t \choose d}{t^2r-d \choose t-d}r(r-1)^{d-1}}{{t^2r \choose t}} < \sum_{d=2}^t \frac{1}{d!} \le \sum_{d=2}^t\frac{1}{2^{d-1}} < 1,$$
which completes the proof.
\end{proof}

In \cite{BP} it was shown that for every tree $T$, we have $\tau_2(T) = \ceil{(5 + \sqrt{1 + 8\Delta(T)})/2}$.  By Proposition \ref{mono_tone}, it thus follows that $\tau_t(T) \ge \ceil{(5 + \sqrt{1 + 8\Delta(T)})/2}$ whenever $t \ge 2$.  In fact this bound is asymptotically best possible, as we show next.

\begin{theorem}\label{tree_upper}
For every positive integer $t$, there exists a constant $c=c(t)$ such that
for every tree $T$ we have $\tau_t(T)\le c\sqrt{\Delta(T)}$.
\end{theorem}
\begin{proof}
Fix a positive integer $t$ and a tree $T$.  Let $k = \sqrt{\Delta(T)}$.  Let
$T^\prime$ be the complete $(\Delta(T)-1)$-ary tree of height $\size{V(T)}$; 
that is, $T^\prime$ is a rooted tree such that all vertices at distance $\size{V(T)}$ from the root are leaves, and all others have $\Delta(T) - 1$ children.  By {\em level $i$} of $T^\prime$ we mean the set of vertices at distance $i$ from the root.
Clearly $T$ is contained in $T^\prime$, so by monotonicity of $\tau_t$ it
suffices to prove that $\tau_t(T^\prime) \le ck$ for some constant $c$ (to be
defined later, but independent of $T$).  Moreover, by Proposition \ref{mono_tone}, we may suppose that $t$ is even.

A {\em palette} is a set of colors.  We color $T^\prime$ using $t+1$ disjoint
palettes, each of size at most $c_1 k$ for some constant $c_1$.  On level $i$ of the
tree we use only those colors in the $i$th palette (with $i$ taken modulo
$t+1$).  This restriction ensures that, whenever $u$ and $v$ are within distance
$t$ of each other, either they lie on the same level of $T^\prime$ or they
receive colors from different palettes (and hence have disjoint labels).  Thus,
we need only consider a single level of $T^\prime$ and show that the vertices
on that level can be colored with at most $c_1 k$ colors.

Within each level, color iteratively with respect to an arbitrary vertex
ordering.  Note that any two vertices on the same level of $T^\prime$ lie at an
even distance.  Fix a vertex $v$ and an integer $d$ between $1$ and $t/2$.   Given a label $L$, say that vertex $u$ {\em forbids} $L$ if $L$ and the label on $u$ have intersection size at least $d(u,v)$.  The
number of vertices at distance $2d$ from $v$, and on the same level as $v$, is
bounded above by $[\Delta(T)]^d$ and hence by $k^{2d}$; each such vertex forbids at
most ${t \choose 2d}{c_1k - 2d \choose t - 2d}$ labels in ${[c_1 k] \choose t}$.  Thus the total number of forbidden labels is at most
$$\sum_{d=1}^{t/2} k^{2d}{t \choose 2d}{c_1k - 2d \choose t - 2d},$$
which is at most
$$k^t \sum_{d=1}^{t/2} \frac{t^{2d}c_1^{t-2d}}{(2d)!(t-2d)!}.$$
We have ${c_1 k \choose t}$ available labels; for fixed $t$ and large $k$, this is
at least $k^t \frac{(c_1-1)^t}{t!}$.  For sufficiently large $c_1$ we have
$$\frac{(c_1-1)^t}{t!} > \sum_{d=1}^{t/2} \frac{t^{2d}c_1^{t-2d}}{(2d)!(t-2d)!},$$
since both sides of the inequality are polynomials in $c_1$, but the left side
has higher degree.  Thus if $c_1$ is large enough, then we can color $v$.
\end{proof}

A graph is {\em $k$-degenerate} if each of its subgraphs contains a vertex of degree at most $k$; trees are precisely the connected 1-degenerate graphs.  For $k \ge 2$, on the class of $k$-degenerate graphs we can improve the bound given by Theorem \ref{general_maxdegree}.

\begin{lemma}\label{degen_dist}
If $G$ is a $k$-degenerate graph, then $G$ has a vertex ordering such that, for
each integer $d\ge 1$ and for each vertex $v$, at most $dk\Delta(G)(\Delta(G) - 1)^{d-2}$ vertices preceding $v$ in the ordering lie at distance $d$ from $v$.
\end{lemma}
\begin{proof}
Construct an ordering of $V(G)$ by repeatedly deleting a vertex $v$ of minimum
degree and prepending $v$ to the ordering.  We claim that this ordering has the desired properties.

Fix $v$ and consider the set of earlier vertices at distance $d$ from $v$.  Each such vertex can be reached from $v$ via a walk of length $d$ in which at least one step moves backward in the ordering.  For each $i$ between 1 and $d$, there are at most $k\Delta(G)(\Delta(G) - 1)^{d-2}$ such walks that move backward on step $i$, since we have at most $k$ choices for the $i$th step, at most $\Delta(G)$ choices for the first, and at most $\Delta(G) -1$ choices for each of the others.
%
\end{proof}

When $d$ is large, the bound in Lemma \ref{degen_dist} is worse than the easy
bound of $\Delta(G)(\Delta(G)-1)^{d-1}$ that holds for all graphs $G$,
regardless of degeneracy.  However, when applying Lemma \ref{degen_dist}, we
will mainly care about small values of $d$.

\begin{theorem} \label{degen_maxdegree}
If $G$ is a $k$-degenerate graph, $k \ge 2$, and $\Delta(G) \le r$, then for
every $t$ we have $\tau_t(G) \le kt + kt^2 r^{1 - 1/t}$.
\end{theorem}
\begin{proof}
Let $c = kt^2r^{1 - 1/t}$.  Let $v_1, \ldots, v_n$ be a vertex ordering of the form guaranteed by Lemma \ref{degen_dist}; we construct a $t$-tone $(c+kt)$-coloring of $G$ by coloring iteratively with respect to this ordering.  

When coloring $v_i$, as many as $kt$ colors may appear on $v_i$'s neighbors; at least $c$ other colors remain.  Thus $v_i$ has at least ${c \choose t}$ candidate labels using these $c$ colors.  As in the proof of Theorem \ref{general_maxdegree}, say that a vertex $u$ {\em forbids} a label $L$ if $L$ and the label on $u$ have intersection of size at least $d(u,v_i)$.  By Lemma \ref{degen_dist}, at most $dkr(r - 1)^{d-2}$ colored vertices lie at distance $d$ from $v_i$; each such vertex forbids at most ${t \choose d}{c-d \choose t-d}$ of the candidates.  
Thus to show that we can color $v_i$, it suffices to show that 
$$\sum_{d=2}^t {t \choose d}{c-d \choose t-d} dkr(r - 1)^{d-2} < {c \choose t},$$
or equivalently, that
$$\sum_{d=2}^t \frac{{t \choose d}{c-d \choose t-d} dkr(r - 1)^{d-2}}{{c \choose t}} < 1.$$

We proceed as in the proof of Theorem \ref{general_maxdegree}.
\begin{align*}
\frac{{t \choose d}{c-d \choose t-d} dkr(r - 1)^{d-2}}{{c \choose t}} &= \frac{t!}{d!(t-d)!} \cdot \frac{(c-d)!}{(t-d)!(c-t)!} \cdot dkr(r-1)^{d-2} \cdot \frac{t!(c-t)!}{c!}\\
                                                                           &= \frac{dk}{d!} \cdot \left ( \frac{t!}{(t-d)!} \right )^2 \cdot \frac{(c-d)!}{c!} \cdot r(r-1)^{d-2}\\
                                                                           &< \frac{k}{(d-1)!} \cdot \frac{(t(t-1)\cdots(t-d+1))^2}{c(c-1)\cdots(c-d+1)}\cdot r^{d-1}\\
                                                                           &= \frac{k}{(d-1)!} \cdot \frac{t^2r^{1-1/d}}{kt^2r^{1 - 1/t}} \cdots \frac{(t-d+1)^2r^{1-1/d}}{kt^2r^{1 - 1/t} - d + 1}\\
                                                                           &\le \frac{1}{(d-1)! k^{d-1}} \cdot \frac{t^2r^{1-1/d}}{t^2r^{1 - 1/t}} \cdots \frac{(t-d+1)^2r^{1-1/d}}{t^2r^{1 - 1/t} - d + 1}
\end{align*}
For $s$ between 0 and $d-1$, we have
$$(t-s)^2r^{1-1/d} \le (t-s)^2r^{1-1/t} = t^2r^{1-1/t} - s(2t-s)r^{1-1/t} \le t^2r^{1-1/t} - s,$$
so
$$\frac{{t \choose d}{c-d \choose t-d} dkr(r - 1)^{d-2}}{{c \choose t}} < \frac{1}{(d-1)!k^{d-1}}.$$
Thus
$$\sum_{d=2}^t \frac{{t \choose d}{c-d \choose t-d} dkr(r - 1)^{d-2}}{{c \choose t}} < \sum_{d=2}^t \frac{1}{(d-1)!k^{d-1}} < 1,$$
as desired.
\end{proof}

Bickle and Phillips \cite{BP} showed that $\tau_2(K_{1,k}) = \Theta(\sqrt{k})$.  Thus by Proposition \ref{mono_tone}, the bound in Theorem \ref{degen_maxdegree} is asymptotically tight (in terms of $\Delta(G)$) when $t = 2$.  

We have made several statements about the asymptotics of $\tau_t(G)$ when
$t$ is fixed and $\Delta(G)$ grows; we now consider what happens when
$\Delta(G)$ is fixed and $t$ grows.  The bound in Theorem
\ref{general_maxdegree} shows that, for fixed values of $\Delta(G)$, we have
$\tau_t(G) \le ct^2$ for some constant $c$.  Our final result shows that the asymptotics of this bound cannot be improved much, if at all.

\begin{theorem} \label{tree_lower}
For each $r \ge 3$, there exists a constant $c$ such that for all $t$, there is a graph $G$ for which $\Delta(G) = r$ and $\tau_t(G) \ge c t^2 / \lg t$. 
\end{theorem}
\begin{proof}
Let $G$ be the complete $(r-1)$-ary tree of height $\ceil{\lg t}$.  
Consider a $t$-tone coloring of $G$; examine the vertices of $G$ in any order.  Since any two vertices of $G$ lie within distance $2\ceil{\lg t}$, each vertex we examine shares fewer than $2\ceil{\lg t}$ colors with each vertex already examined.  Thus, the number of colors used in this coloring is at least
$$\sum_{i=0}^{\size{V(G)} - 1} \max\{0, t - 2 \ceil{\lg t} i\}.$$
When $i \le t / (4 \ceil{\lg t})$, the $i$th term of this sum is
at least $t/2$.  Note that $\size{V(G)} > (r-1)^{\lg t} \ge t > t / (4 \ceil{\lg
t})$, so the sum has at least $t / (4 \ceil{\lg t})$ terms.  Thus, the number of colors used is at least $t^2 / (8 \ceil{\lg t})$.
\end{proof}

\end{document}